\documentclass[11pt,a4paper]{article}

\usepackage{inputenc}
\usepackage{amsmath}
\usepackage{bm}
\usepackage{bbold}
\usepackage{amsthm}

\usepackage{hyperref}

\title{A Complete Closed-Form Solution\\ to a Tropical Extremal Problem\thanks{Advances in Computer Science, Proceedings of the 6th WSEAS European Computing Conference (ECC '12), Prague, Czech Republic, September 24-26, 2012, WSEAS Press, 2012, pp.~146--151. (Recent Advances in Computer Engineering Series, Vol.~5)}
}

\author{Nikolai Krivulin\thanks{Faculty of Mathematics and Mechanics, St.~Petersburg State University, 28 Universitetsky Ave., St.~Petersburg, 198504, Russia, 
nkk@math.spbu.ru}
}

\date{}

\newtheorem{theorem}{Theorem}
\newtheorem{lemma}{Lemma}
\newtheorem{corollary}{Corollary}

\begin{document}

\maketitle

\begin{abstract}
A multidimensional extremal problem in the idempotent algebra setting is considered which consists in minimizing a nonlinear functional defined on a finite-dimensional semimodule over an idempotent semifield. The problem integrates two other known problems by combining their objective functions into one general function and includes these problems as particular cases. A new solution approach is proposed based on the analysis of linear inequalities and spectral properties of matrices. The approach offers a comprehensive solution to the problem in a closed form that involves performing simple matrix and vector operations in terms of idempotent algebra and provides a basis for the development of efficient computational algorithms and their software implementation.
\\

\end{abstract}

\section{Introduction}

The development of solution methods and computational algorithms for solving multidimensional extremal problems is one of the important concerns in the linear tropical (idempotent) algebra \cite{Baccelli92Synchronization,Cuninghamegreen94Minimax,Kolokoltsov97Idempotent,Golan03Semirings,Heidergott05Maxplus,Litvinov07Themaslov,Butkovic10Maxlinear}. The problems under consideration involve the minimization of linear and nonlinear functionals defined on finite-dimensional semimodules over idempotent semifields and may have additional constraints imposed on the feasible solution set in the form of linear tropical equalities and inequalities. Among these problems are idempotent analogues of the linear programming problems \cite{Zimmermann03Disjunctive,Butkovic09Introduction,Butkovic10Maxlinear} and their extensions with nonlinear objective functions \cite{Krivulin05Evaluation,Krivulin06Solution,Krivulin06Eigenvalues,Krivulin09Methods,Krivulin11Algebraic,Krivulin11Analgebraic,Krivulin11Anextremal,Gaubert12Tropical,Krivulin12Anew}. There are solutions to certain problems where both objective function and constraints appear to be nonlinear \cite{Tharwat08Oneclass,Butkovic09Onsome}.

Many extremal problems are formulated and solved only in terms of one idempotent semifield, say the classical semifield $\mathbb{R}_{\max,+}$ in \cite{Butkovic09Introduction,Butkovic10Maxlinear,Gaubert12Tropical}. Some other problems including those considered in \cite{Zimmermann03Disjunctive,Krivulin06Eigenvalues,Krivulin09Methods,Krivulin11Algebraic,Krivulin11Analgebraic,Krivulin11Anextremal,Krivulin12Anew} are treated in a more general setting, which includes the semifield $\mathbb{R}_{\max,+}$ as a particular case. Furthermore, proposed solutions frequently take (see, e.g. \cite{Tharwat08Oneclass,Butkovic09Introduction,Butkovic09Onsome,Butkovic10Maxlinear,Gaubert12Tropical}) the form of an iterative algorithm that produces a solution if any, or indicates that there is no solutions, otherwise. In other cases, as in \cite{Zimmermann03Disjunctive,Krivulin05Evaluation,Krivulin06Eigenvalues,Krivulin09Methods,Krivulin11Algebraic,Krivulin11Analgebraic,Krivulin11Anextremal,Krivulin12Anew}, direct solutions are given in closed form. Finally, note that most of the existing approaches offer some particular solutions rather than give comprehensive solutions to the problems.

In this paper, we consider a multidimensional extremal problem that is a generalization of the problems examined in \cite{Krivulin05Evaluation,Krivulin06Eigenvalues,Krivulin09Methods,Krivulin11Algebraic,Krivulin11Analgebraic,Krivulin11Anextremal,Krivulin12Anew}. Particular cases of the problem arise in various applications including growth rate estimation for the state vector in stochastic dynamic systems with event synchronization \cite{Krivulin05Evaluation,Krivulin09Methods} and single facility location with Chebyshev and rectilinear metrics \cite{Krivulin11Analgebraic,Krivulin11Anextremal}. On the basis of implementation and further development of methods and techniques proposed in \cite{Krivulin05Evaluation,Krivulin06Solution,Krivulin06Eigenvalues,Krivulin09Methods,Krivulin11Algebraic,Krivulin11Analgebraic,Krivulin11Anextremal,Krivulin12Anew,Krivulin12Solution}, we give a complete solution to the problem in a closed form that provides an appropriate basis for both formal analysis and development of efficient computational procedures.

The rest of the paper is as follows. We begin with an introduction to idempotent algebra and outline basic results that underlie the subsequent solutions. Furthermore, examples of tropical extremal problems are presented and their solutions are briefly discussed. A new extremal problem is then introduced and a closed-form solution to the problem under general conditions is established. Finally, solutions to some particular cases and extensions of the problem are given.

\section{Basic Definitions and Results}

We begin with basic algebraic definitions and preliminary results from \cite{Krivulin06Solution,Krivulin06Eigenvalues,Krivulin09Methods} to provide a formal framework for subsequent analysis and solutions presented in the paper. Additional details and thorough investigation of the theory can be found in \cite{Baccelli92Synchronization,Cuninghamegreen94Minimax,Kolokoltsov97Idempotent,Golan03Semirings,Heidergott05Maxplus,Litvinov07Themaslov,Butkovic10Maxlinear}.

\subsection{Idempotent Semifield}

We consider a set $\mathbb{X}$ that is closed under binary operations, addition $\oplus$ and multiplication $\otimes$, and equipped with their related neutral elements, zero $\mathbb{0}$ and identity $\mathbb{1}$. We suppose that the algebraic system $\langle\mathbb{X},\mathbb{0},\mathbb{1},\oplus,\otimes\rangle$ is a commutative semiring with idempotent addition and invertible multiplication. Since for all $\bm{x}\in\mathbb{X}_{+}$, where $\mathbb{X}_{+}=\mathbb{X}\setminus\{\mathbb{0}\}$, there exists its multiplicative inverse $x^{-1}$, the semiring is commonly referred to as the idempotent semifield.

The integer power is introduced in the usual way to represent iterated multiplication. Moreover, we assume that the power with rational exponent is also defined and so consider the semifield to be radicable.

In what follows we omit the multiplication sign $\otimes$ as it is usual in the conventional algebra. The power notation is always used in the above sense.

The idempotent addition naturally induces a partial order on the semifield. Furthermore, we assume that the partial order can be completed to a total order, thus allowing the semifield to be linearly ordered. In the following, the relation signs and the $\min$ symbol are thought of as in terms of this linear order. 

Examples of linearly ordered radicable idempotent semifields include 
\begin{align*}
\mathbb{R}_{\max,+}
&=
\langle\mathbb{R}\cup\{-\infty\},-\infty,0,\max,+\rangle,
\\
\mathbb{R}_{\min,+}
&=
\langle\mathbb{R}\cup\{+\infty\},+\infty,0,\min,+\rangle,
\\
\mathbb{R}_{\max,\times}
&=
\langle\mathbb{R}_{+}\cup\{0\},0,1,\max,\times\rangle,
\\
\mathbb{R}_{\min,\times}
&=
\langle\mathbb{R}_{+}\cup\{+\infty\},+\infty,1,\min,\times\rangle,
\end{align*}
where $\mathbb{R}$ is the set of reals, $\mathbb{R}_{+}=\{x\in\mathbb{R}|x>0\}$.

\subsection{Idempotent Semimodule}

Consider the Cartesian product $\mathbb{X}^{n}$ with column vectors as its elements. A vector with all components equal to $\mathbb{0}$ is called the zero vector and denoted by $\mathbb{0}$. The operations of vector addition $\oplus$ and scalar multiplication $\otimes$ are routinely defined component-wise through the scalar operations introduced on $\mathbb{X}$.

The set $\mathbb{X}^{n}$ with these operations forms a finite-dimensional idempotent semimodule over $\mathbb{X}$.

A vector is called regular if it has no zero components. The set of all regular vectors of order $n$ over $\mathbb{X}_{+}$ is denoted by $\mathbb{X}_{+}^{n}$.

For any nonzero column vector $\bm{x}=(x_{i})\in\mathbb{X}^{n}$ we define a row vector $\bm{x}^{-}=(x_{i}^{-})$, where $x_{i}^{-}=x_{i}^{-1}$ if $x_{i}\ne\mathbb{0}$, and $x_{i}^{-}=\mathbb{0}$ otherwise, $i=1,\ldots,n$.

\subsection{Matrix Algebra}

For conforming matrices with entries from $\mathbb{X}$, addition and multiplication of matrices together with multiplication by scalars follow the conventional rules using the scalar operations defined on $\mathbb{X}$.

A matrix with all entries that are equal to $\mathbb{0}$ is called the zero matrix and denoted by $\mathbb{0}$. A matrix is row (column) regular if it has no zero rows (columns).

Consider the set of square matrices $\mathbb{X}^{n\times n}$. As in the conventional algebra, a matrix is diagonal if its off-diagonal entries are equal to $\mathbb{0}$. A diagonal matrix that has only $\mathbb{1}$ as the diagonal entries is the identity matrix and denoted by $I$. Finally, the exponent notation stands for repeated multiplication for any square matrix $A$ with the obvious condition that $A^{0}=I$. 

The set $\mathbb{X}^{n\times n}$ with the matrix addition and multiplication forms an idempotent semiring with identity.

For any matrix $A=(a_{ij})$ its trace is given by
$$
\mathop\mathrm{tr}A
=
\bigoplus_{i=1}^{n}a_{ii}.
$$

A matrix is called reducible if simultaneous permutations of rows and columns put it into a block-triangular normal form, and irreducible otherwise. The normal form of a matrix $A\in\mathbb{X}^{n\times n}$ is given by
\begin{equation}\label{E-MNF}
A
=
\left(
	\begin{array}{cccc}
		A_{11} 	& \mathbb{0}	& \ldots	& \mathbb{0} \\
		A_{21}	& A_{22}			& 				& \mathbb{0} \\
		\vdots	& \vdots		& \ddots	& \\
		A_{s1}	& A_{s2}		& \ldots	& A_{ss}
	\end{array}
\right),
\end{equation}
where $A_{ii}$ is either irreducible or zero matrix of order $n_{i}$, whereas $A_{ij}$ is an arbitrary matrix of size $n_{i}\times n_{j}$ for all $i=1,\ldots,s$, $j<i$, and $n_{1}+\cdots+n_{s}=n$.

\subsection{Spectrum of Matrices}

Any matrix $A\in\mathbb{X}^{n\times n}$ defines on the semimodule $\mathbb{X}^{n}$ a linear operator with certain spectral properties. Specifically, if the matrix $A$ is irreducible, it has a unique eigenvalue that is given by
\begin{equation}
\lambda
=
\bigoplus_{m=1}^{n}\mathop\mathrm{tr}\nolimits^{1/m}(A^{m})
\label{E-lambda}
\end{equation}
whereas all corresponding eigenvectors are regular.

Let the matrix $A$ be reducible and have the form \eqref{E-MNF}. All eigenvalues of $A$ are among the eigenvalues $\lambda_{i}$ of the diagonal blocks $A_{ii}$, $i=1,\ldots,s$. The value $\lambda=\lambda_{1}\oplus\cdots\oplus\lambda_{s}$ is always an eigenvalue, it is calculated by \eqref{E-lambda} and called the spectral radius of $A$.

\subsection{Linear Inequalities}

Suppose there are a matrix $A\in\mathbb{X}^{m\times n}$ and a regular vector $\bm{d}\in\mathbb{X}_{+}^{m}$. The problem is to solve with respect to the unknown vector $\bm{x}\in\mathbb{X}^{n}$ the linear inequality
\begin{equation}
A\bm{x}
\leq
\bm{d}.
\label{I-Axd}
\end{equation}
 
Clearly, if $A=\mathbb{0}$, then any vector $\bm{x}$ is a solution. Assume now the matrix $A\ne\mathbb{0}$ to have zero columns. It is easy to see that each zero column in $A$ allows the corresponding element of the solution vector $\bm{x}$ to take arbitrary values. The other elements can be found from a reduced inequality with a matrix that is formed by omitting zero columns from $A$ and so becomes column-regular. The solution to the inequality for column-regular matrices is as follows.
 
\begin{lemma}\label{L-IAxd}
A vector $\bm{x}$ is a solution of inequality \eqref{I-Axd} with a column-regular matrix $A$ and a regular vector $\bm{d}$ if and only if
$$
\bm{x}
\leq
(\bm{d}^{-}A)^{-}.
$$
\end{lemma}

For a given square matrix $A\in\mathbb{X}^{n\times n}$ and a vector $\bm{b}\in\mathbb{X}^{n}$, we now find all regular solutions $\bm{x}\in\mathbb{X}_{+}^{n}$ to the inequality
\begin{equation}
A\bm{x}\oplus\bm{b}
\leq
\bm{x}.
\label{I-Axbx}
\end{equation}

To solve the problem, we follow an approach based on the implementation of the function $\mathop\mathrm{Tr}$ that maps each square matrix $A\in\mathbb{X}^{n\times n}$ to a scalar
$$
\mathop\mathrm{Tr}(A)
=
\bigoplus_{m=1}^{n}\mathop\mathrm{tr} A^{m}.
$$

For any matrix $A\in\mathbb{X}^{n\times n}$, we introduce a matrix
$$
A^{\ast}
=
I\oplus A\oplus\cdots\oplus A^{n-1}.
$$

Assume a matrix $A$ to be represented in its normal form \eqref{E-MNF}. We define a diagonal matrix
$$
D
=
\left(
	\begin{array}{ccc}
		A_{11}		&					& \mathbb{0} \\
							& \ddots	& \\
		\mathbb{0}&					& A_{ss}
	\end{array}
\right),
$$
and a low triangular matrix 
$$
T
=
\left(
	\begin{array}{cccc}
		\mathbb{0}& \ldots	& \ldots		& \mathbb{0} \\
		A_{21}		& \ddots	& 					& \vdots \\
		\vdots		& \ddots	& \ddots		& \vdots \\
		A_{s1}		& \ldots	& A_{s,s-1}	& \mathbb{0}
	\end{array}
\right),
$$
which present the diagonal and triangular parts of the decomposition of $A$ in the form
\begin{equation}
A
=
D\oplus T.
\label{E-ADT}
\end{equation}

Note that if the matrix $A$ is irreducible, we put $D=A$ and $T=\mathbb{0}$.

\begin{theorem}\label{T-IAxbx}
Let $\bm{x}$ be the general regular solution of inequality \eqref{I-Axbx} with a matrix $A$ in the form of \eqref{E-ADT}. Then the following statements are valid:
\begin{enumerate}
\item If $\mathop\mathrm{Tr}(A)\leq\mathbb{1}$, then $\bm{x}=(D^{\ast}T)^{\ast}D^{\ast}\bm{u}$ for all $\bm{u}\in\mathbb{X}_{+}^{n}$ such that $\bm{u}\geq\bm{b}$.
\item If $\mathop\mathrm{Tr}(A)>\mathbb{1}$, then there is no regular solution.
\end{enumerate}
\end{theorem}

\section{Tropical Extremal Problems}

We now turn to the discussion of multidimensional extremal problems formulated in terms of idempotent algebra. The problems are established to minimize both linear and nonlinear functionals defined on semimodules over idempotent semifields, subject to constraints in the form of linear equalities and inequalities.

In this section, the symbols $A$ and $C$ stand for given matrices, $\bm{b}$, $\bm{d}$, $\bm{p}$, $\bm{q}$, $\bm{g}$ and $\bm{h}$ for vectors, and $r$ and $s$ for numbers. We start with an idempotent analogue of linear programming problems examined in \cite{Butkovic09Introduction,Butkovic10Maxlinear} and defined in terms of the semifield $\mathbb{R}_{\max,+}$ to find the solution $\bm{x}$ to the problem
\begin{gather*}
\min\ (\bm{p}^{T}\bm{x}\oplus r),
\\
A\bm{x}\oplus\bm{b}\leq C\bm{x}\oplus\bm{d}.
\end{gather*}

A solution technique that is based on an iterative algorithm and called the alternating method is proposed which produces a solution if any, or indicates that there is no solution otherwise.

The technique is extended in \cite{Gaubert12Tropical} to provide an iterative computational scheme for a problem with nonlinear objective function given by
\begin{gather*}
\min\ (\bm{p}^{T}\bm{x}\oplus r)(\bm{q}^{T}\bm{x}\oplus s)^{-1},
\\
A\bm{x}\oplus\bm{b}\leq C\bm{x}\oplus\bm{d}.
\end{gather*}

There are certain problems which can be solved directly in a closed form. Specifically, an explicit formula is proposed in \cite{Zimmermann03Disjunctive} within the framework of optimization of max-separable functions under disjunctive constraints for the solution of the problem
\begin{gather*}
\min\ \bm{p}^{T}\bm{x},
\\
C\bm{x}\geq\bm{b},
\\
\bm{g}\leq\bm{x}\leq\bm{h}.
\end{gather*}

Furthermore, in \cite{Krivulin05Evaluation,Krivulin06Eigenvalues,Krivulin09Methods,Krivulin11Algebraic,Krivulin11Analgebraic}, a problem is examined which is to find regular solutions $\bm{x}$ that provide
$$
\min\ \bm{x}^{-}A\bm{x}.
$$

To get a closed form solution to the problem, an approach is applied that uses results of the spectral theory of linear operators in idempotent algebra. 

Finally, a closed form solution based on a technique of solving linear equations and inequalities is derived in \cite{Krivulin12Anew} for the problem
\begin{gather*}
\min\ (\bm{x}^{-}\bm{p}\oplus\bm{q}^{-}\bm{x}),
\\
A\bm{x}\leq\bm{x}.
\end{gather*}

In the rest of the paper, we consider a problem with a general objective function that actually contains the objective functions of the last two problems as particular cases. For the problem when there are no additional constraints imposed on the solution, a general solution is given in a closed form.

\section{A New General Extremal Problem}

Given a matrix $A\in\mathbb{X}^{n\times n}$ and vectors $\bm{p},\bm{q}\in\mathbb{X}^{n}$, consider the problem to find $\bm{x}$ that provides
\begin{equation}
\min_{\bm{x}\in\mathbb{X}_{+}^{n}}(\bm{x}^{-}A\bm{x}\oplus\bm{x}^{-}\bm{p}\oplus\bm{q}^{-}\bm{x}).
\label{P-xAxxpqx}
\end{equation}

A complete explicit solution to the problem under general conditions as well as to some particular cases and extensions is given in the subsequent sections.

\subsection{The Main Result}

We start with a solution to the problem in a general setting that is appropriate for many applications.

\begin{theorem}\label{T-xAxxpqx}
Suppose $A$ is a matrix in the form \eqref{E-MNF}, $\bm{p}$ is a vector, $\bm{q}$ is a regular vector, $\lambda$ is the spectral radius of $A$, and
$$
\Delta
=
(\bm{q}^{-}\bm{p})^{1/2},
\qquad
\mu
=
\lambda\oplus\Delta\ne\mathbb{0}.
$$

Define a matrix
$$
A_{\mu}
=
\mu^{-1}A
=
D_{\mu}\oplus T_{\mu},
$$
where $D_{\mu}$ and $T_{\mu}$ are respective diagonal and lower triangular parts of $A_{\mu}$, and a matrix
$$
B
=
(D_{\mu}^{\ast}T_{\mu})^{\ast}D_{\mu}^{\ast}.
$$

Then the minimum in \eqref{P-xAxxpqx} is equal to $\mu$ and attained if and only if
$$
\bm{x}
=
B\bm{u}
$$
for all regular vectors $\bm{u}$ such that
$$
\mu^{-1}\bm{p}
\leq
\bm{u}
\leq
\mu(\bm{q}^{-}B)^{-}.
$$
\end{theorem}
\begin{proof}
We show that both $\lambda$ and $\Delta$ are lower bounds for the objective function in \eqref{P-xAxxpqx}, and then get all regular vectors $\bm{x}$ that yield the value $\mu=\lambda\oplus\Delta$ of the function. To verify that $\lambda$ is a lower bound, we write
$$
\bm{x}^{-}A\bm{x}\oplus\bm{x}^{-}\bm{p}\oplus\bm{q}^{-}\bm{x}
\geq
\bm{x}^{-}A\bm{x}.
$$

Assume the matrix $A$ to be irreducible and $\lambda$ to be its unique eigenvalue. We take a corresponding eigenvector $\bm{x}_{0}$ and note that for all $\bm{x}\in\mathbb{X}_{+}^{n}$,
$$
\bm{x}_{0}^{-}\bm{x}_{0}
=
\mathbb{1},
\qquad
\bm{x}\bm{x}_{0}^{-}
\geq(\bm{x}^{-}\bm{x}_{0})^{-1}I.
$$

Furthermore, we have
\begin{multline*}
\bm{x}^{-}A\bm{x}
=
\bm{x}^{-}A\bm{x}\bm{x}_{0}^{-}\bm{x}_{0}
=
\bm{x}^{-}A(\bm{x}\bm{x}_{0}^{-})\bm{x}_{0}
\\
\geq
\bm{x}^{-}A\bm{x}_{0}(\bm{x}^{-}\bm{x}_{0})^{-1}
=
\lambda\bm{x}^{-}\bm{x}_{0}(\bm{x}^{-}\bm{x}_{0})^{-1}
=
\lambda.
\end{multline*}

Consider an arbitrary matrix $A$ taking the form \eqref{E-MNF}. Any vector $\bm{x}$ now admits a decomposition into subvectors $\bm{x}_{1},\ldots,\bm{x}_{s}$ according to the decomposition of $A$ into column blocks. With the above result for irreducible matrices, we obtain
$$
\bm{x}^{-}A\bm{x}
=
\bigoplus_{i=1}^{s}\bigoplus_{j=1}^{i}
\bm{x}_{i}^{-}A_{ij}\bm{x}_{j}
\geq
\bigoplus_{i=1}^{s}
\bm{x}_{i}^{-}A_{ii}\bm{x}_{i}
\geq
\bigoplus_{i=1}^{s}\lambda_{i}
=
\lambda.
$$

Now we show that $\Delta=(\bm{q}^{-}\bm{p})^{1/2}$ is also a lower bound for the objective function. We have
$$
\bm{x}^{-}A\bm{x}\oplus\bm{x}^{-}\bm{p}\oplus\bm{q}^{-}\bm{x}
\geq
\bm{x}^{-}\bm{p}\oplus\bm{q}^{-}\bm{x}.
$$

Let us take any vector $\bm{x}\in\mathbb{X}_{+}^{n}$ and denote
$$
r
=
\bm{x}^{-}\bm{p}\oplus\bm{q}^{-}\bm{x}.
$$

The last equality leads to two inequalities 
$$
r
\geq
\bm{q}^{-}\bm{x}
>
\mathbb{0},
\qquad
r
\geq
\bm{x}^{-}\bm{p}.
$$

Multiplication of the first inequality by $r^{-1}\bm{x}^{-}$ from the right gives $\bm{x}^{-}\geq r^{-1}\bm{q}^{-}\bm{x}\bm{x}^{-}\geq r^{-1}\bm{q}^{-}$. Substitution of $\bm{x}^{-}\geq r^{-1}\bm{q}^{-}$ into the second results in $r\geq r^{-1}\bm{q}^{-}\bm{p}=r^{-1}\Delta^{2}$, whence it follows that
$$
\bm{x}^{-}\bm{p}\oplus\bm{q}^{-}\bm{x}
=
r
\geq
\Delta.
$$

By combining both bounds, we conclude that
$$
\bm{x}^{-}A\bm{x}\oplus\bm{x}^{-}\bm{p}\oplus\bm{q}^{-}\bm{x}
\geq
\lambda\oplus\Delta
=
\mu.
$$

It remains to find all regular solutions $\bm{x}$ of the equation
$$
\bm{x}^{-}A\bm{x}\oplus\bm{x}^{-}\bm{p}\oplus\bm{q}^{-}\bm{x}
=
\mu.
$$

Since $\bm{x}^{-}A\bm{x}\oplus\bm{x}^{-}\bm{p}\oplus\bm{q}^{-}\bm{x}\geq\mu$ for all $\bm{x}\in\mathbb{X}_{+}^{n}$, the set of regular solutions of the equation coincides with that of the inequality
$$
\bm{x}^{-}A\bm{x}\oplus\bm{x}^{-}\bm{p}\oplus\bm{q}^{-}\bm{x}
\leq
\mu,
$$
which itself is equivalent to the system of inequalities
\begin{align}
\bm{x}^{-}A\bm{x}\oplus\bm{x}^{-}\bm{p}
&\leq
\mu,
\label{I-xAxxpmu}
\\
\bm{q}^{-}\bm{x}
&\leq
\mu.
\label{I-qxmu}
\end{align}

Let us consider inequality \eqref{I-xAxxpmu}. After multiplication of the inequality by $\mu^{-1}\bm{x}$ from the left, we write
$$
A_{\mu}\bm{x}\oplus\mu^{-1}\bm{p}\leq\mu^{-1}\bm{x}\bm{x}^{-}A\bm{x}\oplus\mu^{-1}\bm{x}\bm{x}^{-}\bm{p}\leq\bm{x},
$$
and then arrive at the inequality
$$
A_{\mu}\bm{x}\oplus\mu^{-1}\bm{p}
\leq
\bm{x}.
$$

On the other hand, left multiplication of the obtained inequality by $\mu\bm{x}^{-}$ directly yields inequality \eqref{I-xAxxpmu}, and thus both inequalities are equivalent.

Since $\mathop\mathrm{Tr}(A_{\mu})=\mathop\mathrm{Tr}(\mu^{-1}A)\leq\mathbb{1}$, we can apply Theorem~\ref{T-IAxbx} to the last inequality so as to get the general solution of inequality \eqref{I-xAxxpmu} in the form
$$
\bm{x}
=
(D_{\mu}^{\ast}T_{\mu})^{\ast}D_{\mu}^{\ast}\bm{u}
=
B\bm{u},
$$
where $\bm{u}\in\mathbb{X}_{+}^{n}$ is any vector such that $\bm{u}\geq\mu^{-1}\bm{p}$.

Substitution of the solution into inequality \eqref{I-qxmu} gives an inequality $\bm{q}^{-}B\bm{u}\leq\mu$. Application of Lemma~\ref{L-IAxd} to the last inequality yields $\bm{u}\leq\mu(\bm{q}^{-}B)^{-}$.

By combining lower and upper bounds obtained for the vector $\bm{u}$, we finally arrive at the solution
$$
\bm{x}
=
B\bm{u},
$$
for all $\bm{u}\in\mathbb{X}_{+}^{n}$ such that
$$
\mu^{-1}\bm{p}
\leq
\bm{u}
\leq
\mu(\bm{q}^{-}B)^{-}.
\qedhere
$$
\end{proof}

\subsection{Particular Cases and Extensions}

Consider problem \eqref{P-xAxxpqx} with an irreducible matrix $A$. Since in this case $D_{\mu}=A_{\mu}$, $T_{\mu}=\mathbb{0}$, and $B=A_{\mu}^{\ast}$, the statement of Theorem~\ref{T-xAxxpqx} takes a reduced form.
\begin{corollary}
If $A$ is an irreducible matrix, then the solution set of \eqref{P-xAxxpqx} is given by
$$
\bm{x}
=
A_{\mu}^{\ast}\bm{u},
$$
for all $\bm{u}\in\mathbb{X}_{+}^{n}$ such that
$$
\mu^{-1}\bm{p}
\leq
\bm{u}
\leq
\mu(\bm{q}^{-}A_{\mu}^{\ast})^{-}.
$$
\end{corollary}

Specifically, when $A=\mathbb{0}$, we have $B=A_{\mu}^{\ast}=I$ and $\mu=\Delta$. The solution set is further reduced to
$$
\Delta^{-1}\bm{p}
\leq
\bm{x}
\leq
\Delta\bm{q},
$$
which coincides with that in \cite{Krivulin12Anew}.

Suppose the vector $\bm{q}$ in problem \eqref{P-xAxxpqx} is irregular. In this case, the matrix $\bm{q}^{-}B$ in the inequality
$$
\bm{q}^{-}B\bm{u}\leq\mu
$$
may be not column-regular, which prevents direct application of Lemma~\ref{L-IAxd} as in Theorem~\ref{T-xAxxpqx}.

Let $J=\mathop\mathrm{supp}(\bm{q}^{-}B)$ be the set of indices of nonzero elements in the row vector $\bm{q}^{-}B$. Denote by $(\bm{q}^{-}B)_{J}$ and $\bm{u}_{J}$ subvectors that have only components with indices from $J$. The solution of the above inequality is given by the constraints $\bm{u}_{J}\leq\mu(\bm{q}^{-}B)_{J}^{-}$ for the subvector $\bm{u}_{J}$, whereas the rest components of the vector $\bm{u}$ can take arbitrary values.

Now we can somewhat weaken conditions of Theorem~\ref{T-xAxxpqx} as follows. 
\begin{theorem}
Under the assumptions of Theorem~\ref{T-xAxxpqx}, let $\bm{q}\ne\mathbb{0}$ be arbitrary vector and $J=\mathop\mathrm{supp}(\bm{q}^{-}B)$.

Then the minimum in \eqref{P-xAxxpqx} is equal to $\mu$ and attained if and only if
$$
\bm{x}
=
B\bm{u}
$$
for all regular vectors $\bm{u}$ such that
$$
\mu^{-1}\bm{p}
\leq
\bm{u},
\quad
\bm{u}_{J}
\leq
\mu(\bm{q}^{-}B)_{J}^{-}.
$$
\end{theorem}

Finally note, that when $\bm{q}=\mathbb{0}$ we have $J=\emptyset$ and so the upper bound for $\bm{u}$ disappears.

\section{Conclusion}

A complete closed-form solution has been derived for a tropical extremal problem with nonlinear objective function and without constraints. The solution actually involves performing simple matrix and vector operations in terms of idempotent algebra and provides a basis for the development of efficient computational algorithms and their software implementation.

As a suggested line of further research, solutions to the problems under constraints in the form of tropical linear equalities and inequalities are to be considered. Practical examples of successful application of the results obtained are also of great interest.

\bibliography{KrivulinN}

\end{document}